\documentclass[11pt]{scrartcl}
\usepackage{amsfonts,amssymb,amsthm,amsmath,booktabs}
\usepackage{color,xcolor,graphicx}
\usepackage[latin1]{inputenc}
\usepackage{caption}
\usepackage{subcaption}

\ifx\pdftexversion\undefined
\usepackage{nohyperref}
\else
\usepackage[colorlinks=true,linkcolor={blue},citecolor={blue},urlcolor={blue},
  pdfauthor={Thomas Apel, Serge Nicaise, Johannes Pfefferer},
  pdfstartview={Fit}]{hyperref}
\fi

\newcommand{\abssec}[1]{\noindent\normalsize {\bfseries #1\quad }\ignorespaces}
\renewenvironment{abstract}{\abssec{Abstract}}{\par\vspace{.1in}}
\newenvironment{keywords}{\abssec{Key Words}}{\par\vspace{.1in}}
\newenvironment{AMSMOS}{\abssec{AMS subject
  classification}}{\par\vspace{.1in}}
\newenvironment{acknowledgement}{\abssec{Acknowledgement}}{\par\vspace{.1in}}

\theoremstyle{plain}
\newtheorem{theorem}{Theorem}[section]
\newtheorem{corollary}[theorem]{Corollary}

\newtheorem{lemma}[theorem]{Lemma}

\theoremstyle{definition}
\newtheorem{remark}[theorem]{Remark}

\numberwithin{equation}{section}

\def\R{\mathbb{R}}

\newcommand{\normaltrace}{\zeta}

\definecolor{darkgreen}{rgb}{0.0, 0.5, 0.3}

\begin{document}

\title{\LARGE Discretization of the Poisson equation \\ with non-smooth data and\\
		emphasis on non-convex domains\thanks{The work
    was partially supported by Deutsche Forschungsgemeinschaft,
    priority program 1253 and IGDK 1754.}}

\author{Thomas Apel\thanks{\texttt{thomas.apel@unibw.de},
    Universit\"at der Bundeswehr M\"unchen, Institut f\"ur Mathematik
    und Bauinformatik, D-85579 Neubiberg, Germany} \and Serge
  Nicaise\thanks{\texttt{snicaise@univ-valenciennes.fr}, LAMAV,
    Institut des Sciences et Techniques de Valenciennes, Universit\'e
    de Valenciennes et du Hainaut Cambr\'esis, B.P. 311, 59313
    Valenciennes Cedex, France } \and Johannes
  Pfefferer\thanks{\texttt{johannes.pfefferer@unibw.de}, Universit\"at
    der Bundeswehr M\"unchen, Institut f\"ur Mathematik und
    Bauinformatik, D-85579 Neubiberg, Germany}}
\maketitle

\begin{abstract}
  Several approaches are discussed how to understand the solution of
  the Dirichlet problem for the Poisson equation when the Dirichlet
  data are non-smooth such as if they are in $L^2$ only. For the method of transposition (sometimes
  called \emph{very weak formulation}) three spaces for the test
  functions are considered, and a regularity result is proved. An
  approach of Berggren is recovered as the method of transposition
  with the second variant of test functions. A further concept is the
  regularization of the boundary data combined with the weak solution
  of the regularized problem. The effect of the regularization error
  is studied.

  The regularization approach is the simplest  to discretize. The
  discretization error is estimated for a sequence of quasi-uniform
  meshes. Since this approach turns out to be equivalent to Berggren's
  discretization his error estimates are rendered more precisely.
  Numerical tests show that  the error estimates are sharp, in
  particular that the order becomes arbitrarily small when the maximal
  interior angle of the domain tends to $2\pi$.
\end{abstract}


\begin{AMSMOS}
  65N30; 65N15
\end{AMSMOS}

\section{Introduction}

The motivation for this paper is to consider the boundary value problem
\begin{align} \label{eq:bvp}
  -\Delta y &= f \quad\text{in }\Omega, \qquad 
  y = u \quad\text{on }\Gamma:=\partial\Omega,
\end{align}
with right hand side $f\in H^{-1}(\Omega)$ and boundary data $u\in L^2(\Gamma)$. We assume $\Omega\subset\R^2$ to be a
bounded polygonal domain with boundary $\Gamma$. Such problems arise
in optimal control when the Dirichlet boundary control is considered
in $L^2(\Gamma)$ only, see for example the papers by
Deckelnick, G\"unther, and Hinze, \cite{DeckelnickGuentherHinze2009},
French and King, \cite{FrenchKing1991}, May, Rannacher, and
Vexler, \cite{MayRannacherVexler2008}, and Apel, Mateos, Pfefferer, and R\"osch, \cite{ApelMateosPfeffererRoesch2013}.
On the continuous level we even admit more irregular data.

In Section \ref{sect:bvp} we analyze several ways how to understand
the solution of the boundary value problem \eqref{eq:bvp} for which we cannot expect a weak solution $y\in H^1(\Omega)$.
The most popular method to solve problem \eqref{eq:bvp} is the transposition method
that goes back to Lions and Magenes \cite{LionsMagenes1968} and that
is based on the use of some integration by parts. 
This formally leads
to the \emph{very weak formulation}: Find $y\in Y$ such that
\begin{align}\label{eq:veryweakintro}
   (y, \Delta v)_\Omega =
  (u,\partial_n v)_\Gamma - (f,v)_\Omega \quad\forall
  v\in V
\end{align}
with $(w,v)_G:=\int_G wv$ denoting the $L^2(G)$ scalar product
or an appropriate duality product.  The main issue is to
find the appropriate trial space $Y$ and test space $V$.  In the
convex case, it turns out that a good choice is $Y=L^2(\Omega)$ and
$V=H^2(\Omega)\cap H^1_0(\Omega)$. This case is investigated by many
authors, some of them are mentioned in Subsection
\ref{sec:methodtransposition}, but we will see that such a choice is
not appropriate in the non-convex case (loss of uniqueness) and
present some remedies (enlarged test spaces) in Subsections \ref{sec:general} and \ref{sec:weighted}.
We will further show in Subsection \ref{sec:general} that in the correct setting
the very weak solution corresponds to the one obtained by an
integral equation technique and hence has $H^{1/2}(\Omega)$
regularity.

The main drawback of the very weak formulation is the fact that a
conforming discretization of the test space should be made by $C^1$-elements.  Hence
Berggren proposed in \cite{Berggren2004} to introduce two new variables,
$\varphi:=-\Delta v$ and $\normaltrace:=\partial_n v$
allowing to perform a simpler numerical analysis,
see Subsection \ref{sec:Berggren}. In Subsection \ref{sec:regularizationstrategy} we
propose another method that consists of regularizing the Dirichlet
datum $u$ by approximating it by a sequence of functions $u^h$ in
$H^{1/2}(\Gamma)$ using for example an interpolation
operator.  This allows to compute a sequence of weak solutions $y^h\in
H^1(\Omega)$, and we show that they converge to the very weak solution
with an explicit convergence rate.

A negative result about the
well-posedness of the weak formulation with $L^1(\Gamma)$-data completes the
discussion on the continuous level.

Section \ref{sect:num} is devoted to the numerical analysis.
We start with Berggren's numerical approach and recall
his error estimates in Subsection \ref{sec:numberg} for completeness.
Next we perform in Subsection \ref{sec:numreg} a numerical analysis of our
regularization approach and prove error estimates for the piecewise
linear approximation on a family of conforming, quasi-uniform finite
element meshes. Notice that it turns out that on the discrete level
Berggren's approach is a particular case of our regularization strategy.
The convergence order is $\frac12$ in the convex case but smaller in the non-convex case. This reduction
can be explained by the singular behaviour of the solution of the dual problem. In our paper
\cite{ApelNicaisePfefferer2014b} we investigate the singular
complement method to remedy the suboptimality of the standard finite
element method in non-convex domains.


Finally, in Section \ref{sect:test} we present numerical tests in
order to illustrate that our error estimates are sharp.
The paper ends with some remarks about the three-dimensional case
and about data with different regularity than assumed above.

\section{\label{sect:bvp}Analysis of the boundary value problem}

In this section we analyze several ways how to understand the solution
of the boundary value problem \eqref{eq:bvp} for which we cannot
expect a weak solution $y\in H^1(\Omega)$. 

For keeping the notation succinct we assume that the polygonal domain
$\Omega$ has at most one non-convex corner with interior angle, denoted by $\omega$.
Let $r,\theta$ be the corresponding polar coordinates and define
$\lambda:=\pi/\omega$. The boundary segments of $\Omega$ are denoted
by $\Gamma_j$, $j=1,\ldots,N$.


\subsection{\label{sec:methodtransposition}Method of transposition in convex polygonal domains}

The method of transposition goes back at least to Lions and Magenes,
\cite{LionsMagenes1968}, and is used by several other authors including
French and King, \cite{FrenchKing1991}, Casas and Raymond,
\cite{CasasRaymond2006}, Dec\-kel\-nick, G\"unther, and Hinze,
\cite{DeckelnickGuentherHinze2009}, and May, Rannacher, and Vexler,
\cite{MayRannacherVexler2008}. Since by partial integration the
derivation
\begin{align*}
  (f,v)_\Omega &= -(\Delta y,v)_\Omega =
  (\nabla y,\nabla v)_\Omega &&\text{for } v\in H^1_0(\Omega)\\
  &= (y,\partial_n v)_\Gamma - (y, \Delta v)_\Omega
  &&\text{for } v\in H^2(\Omega)\cap H^1_0(\Omega)
\end{align*}
is valid we get the \emph{very weak formulation}: Find
\begin{align}\label{eq:veryweak1}
  y\in L^2(\Omega):\quad (y, \Delta v)_\Omega =
  (u,\partial_n v)_\Gamma - (f,v)_\Omega \quad\forall
  v\in  H^2(\Omega)\cap H^1_0(\Omega).
\end{align}
May, Rannacher, and Vexler, \cite{MayRannacherVexler2008}, proved the
existence of a solution in $L^2(\Omega)$ in the case of a convex
polygonal domain $\Omega$.

\begin{lemma}\label{lem:MRV}
  If the domain $\Omega$ is convex there exists a unique solution
  $y\in L^2(\Omega)$ of problem \eqref{eq:veryweak1} that satisfies
  the a priori error estimate
  \begin{align}\label{eq:MRV}
    \|y\|_{L^2(\Omega)} \le c 
    \left(\|u\|_{\prod_{j=1}^N H^{1/2}_{00}(\Gamma_j)'} + \|f\|_{\left(H^2(\Omega)\cap H^1_0(\Omega)\right)'} \right)
  \end{align}
  provided that $u\in \prod_{j=1}^N H^{1/2}_{00}(\Gamma_j)'$ and $f\in
  \left(H^2(\Omega)\cap H^1_0(\Omega)\right)'$.
\end{lemma}

Recall that $H^{1/2}_{00}(\Gamma_j)$ is the space of functions whose
extension by zero to  $\Gamma$ is in $H^{1/2}(\Gamma)$.

\begin{proof}
  For being self-contained we sketch here the proof of \cite[Lemma
  2.1]{MayRannacherVexler2008}. The idea is first to assume more
  regular data, $u\in H^{1/2}(\Gamma)$, $f\in H^{-1}(\Omega)$ such
  that a weak solution $y\in H^1(\Omega)$ exists, then to show
  \eqref{eq:MRV}, and finally to use a standard density argument since
  $H^2(\Omega)\cap H^1_0(\Omega) \stackrel{c}{\hookrightarrow}
  H^1_0(\Omega)$ and $H^{1/2}(\Gamma) \stackrel{c}{\hookrightarrow}
  L^2(\Gamma) = \prod_{j=1}^N L^2(\Gamma_j)
  \stackrel{c}{\hookrightarrow} \prod_{j=1}^N
  H^{1/2}_{00}(\Gamma_j)'$.

  The estimate \eqref{eq:MRV} is proven by using a duality argument.
  Due to the convexity of the domain there exists a solution $w\in
  H^2(\Omega)$ of the auxiliary problem
  \begin{align} \label{eq:auxbvpMRV}
    -\Delta w &= y \quad\text{in }\Omega, \qquad
    w = 0 \quad\text{on }\Gamma,
  \end{align}
  such that
  \begin{align*}
    \|y\|_{L^2(\Omega)}^2 & = (y,-\Delta w)_\Omega = 
    (f,w)_\Omega - 
    ( u,\partial_n w)_{\Gamma}\\
    &\le c\left( \|u\|_{\prod_{j=1}^N H^{1/2}_{00}(\Gamma_j)'} + 
    \|f\|_{\left(H^2(\Omega)\cap H^1_0(\Omega)\right)'} \right)\|w\|_{H^2(\Omega)},
  \end{align*}
  where one uses that the mapping $H^2(\Omega)\cap H^1_0(\Omega) \to
  \prod_{j=1}^N H^{1/2}_{00}(\Gamma_j)$, $w\mapsto\partial_nw$, is
  surjective due to \cite[Thm. 1.5.2.8]{grisvard:85a}. The desired
  estimated \eqref{eq:MRV} is then obtained by using the a priori
  estimate $\|w\|_{H^2(\Omega)} \le c \|y\|_{L^2(\Omega)}$ and by
  division by $\|y\|_{L^2(\Omega)}$.
\end{proof}

The method of proof of this lemma will be revisited in Section \ref{sec:regularizationstrategy},
where we start the discussion of the regularization approach.
However, there we will work with different function spaces such that non-convex domains are included in the theory as well.
The given proof of Lemma \ref{lem:MRV} is even restricted to convex domains since the isomorphism 
\[
  \Delta w\in L^2(\Omega),\ w_{|\Gamma}=0\quad\Leftrightarrow\quad w\in H^2(\Omega)\cap H^1_0(\Omega)
\]
is used.
If the domain is non-convex one loses at least the uniqueness
of the solution $y$ of \eqref{eq:veryweak1}.  For example, take
$\Omega=\{(r\cos\theta,r\sin\theta)\in\R^2: 0<r<1,
0<\theta<\omega\}$ and $\lambda=\pi/\omega$, then both
$y_1=r^\lambda\sin(\lambda\theta)$ and
$y_2=r^{-\lambda}\sin(\lambda\theta)$ are harmonic in $\Omega$ and
$y_1=y_2$ on $\Gamma$. Both satisfy \eqref{eq:veryweak1} with $f\equiv0$ and $g=y_1=y_2$ on $\Gamma$. Hence, one
needs a larger test space in order to rule out $y_2$.

\subsection{Method of transposition in general polygonal domains}\label{sec:general}
In a first instance we replace the test space
$H^2(\Omega)\cap H^1_0(\Omega)$ by
\begin{align}\label{eq:H1Delta}
  V:=H^1_0(\Omega)\cap H^1_\Delta(\Omega) \quad\text{with}\quad
  H^1_\Delta(\Omega):=\{v\in H^1(\Omega): \Delta v\in L^2(\Omega)\}.
\end{align}
Since $|v|_{H^1(\Omega)} \le c\|\Delta
v\|_{L^2(\Omega)}$, the graph norm in $V$, that is $\|\Delta
v\|_{L^2(\Omega)} + |v|_{H^1(\Omega)}$, is equivalent to $\|\Delta
v\|_{L^2(\Omega)}$ such that we will use henceforth
\[
  \|v\|_V=\|\Delta v\|_{L^2(\Omega)}.
\]
Furthermore, let us denote by $V_\Gamma$ the space of normal derivatives $\partial_n v$ of functions $v\in V$.
According to \cite[Theorem 1.5.3.10]{grisvard:85a} this space is well-defined and a subspace of $\prod_{j=1}^N H^{1/2}_{00}(\Gamma_j)'$.
The natural norm in $V_\Gamma$ is given by
\[
 \|g\|_{V_\Gamma}:=\inf\left\{\|v\|_V:v\in V,\partial_nv=g\right\}.
\]
In particular, we have for $v\in V$ that
\begin{equation}\label{jonny:normalV}
  \|\partial_n v\|_{V_\Gamma}\leq \|v\|_{V}.
\end{equation}
Since the previous definitions of the spaces $V$ and $V_\Gamma$ are rather formal let us discuss the structure of these spaces.
\begin{remark}
	The spaces $V$ and $V_\Gamma$ can be characterized as follows:
 \begin{enumerate}
  \item If $\Omega$ is convex the spaces $V$ and $H^2(\Omega)\cap H^1_0(\Omega)$ coincide. However, in non-convex domains the situation is different. In this case there is the splitting
  \[
   V=\left(H^2(\Omega)\cap H^1_0(\Omega)\right)\oplus \operatorname{Span} \{\xi(r)\,r^\lambda\sin(\lambda\theta)\},
  \]
  where $\xi$ denotes a smooth cut-off function which is equal to one in the neighborhood of the non-convex corner. For more details we refer to \cite[Sections 1.5, 2.3 and 2.4]{grisvard:92b} and \cite[Theorem 4.4.3.7]{grisvard:85a}.
	\item The mapping $H^2(\Omega)\cap H^1_0(\Omega) \to
  \prod_{j=1}^N H^{1/2}_{00}(\Gamma_j)$, $w\mapsto\partial_nw$, is
  surjective due to \cite[Thm. 1.5.2.8]{grisvard:85a}.
  Accordingly, in the convex case $V_\Gamma$ is just $\prod_{j=1}^N H^{1/2}_{00}(\Gamma_j)$, whereas in the non-convex case there holds
  \[
   V_\Gamma=\left(\prod_{j=1}^N H^{1/2}_{00}(\Gamma_j)\right)\oplus \operatorname{Span} \{\xi(r)\,r^{\lambda-1}\}.
  \]
  \item In the non-convex case there is $V\hookrightarrow H^s(\Omega)\cap H^1_0(\Omega)$ for $s<1+\lambda$ and $V_\Gamma\hookrightarrow H^t(\Gamma)$ for $t<\lambda-\frac12$.
  This implies $\left(H^{s}(\Omega)\cap H^1_0(\Omega)\right)'\hookrightarrow V'$ and $H^t(\Gamma)'\hookrightarrow V_\Gamma'$.
 \end{enumerate}
\end{remark}

\begin{lemma}\label{lem:existenceveryweak}
	Let $f\in V'$ and $u\in V_\Gamma'$. Then there exists a unique solution
	\begin{align}\label{eq:veryweak2}
		y\in L^2(\Omega):\quad (y, \Delta v)_\Omega =
		(u,\partial_n v)_\Gamma - (f,v)_\Omega \quad\forall
		v\in V
	\end{align}
  with
  \[
    \|y\|_{L^2(\Omega)} \le  
    \|u\|_{V_\Gamma'} + \|f\|_{V'}.
  \]
\end{lemma}
\
\begin{proof}
	The proof of this lemma is based on the Babu\v ska--Lax--Milgram theorem.
  Due to \eqref{jonny:normalV} the right hand side of \eqref{eq:veryweak2} defines a
  linear functional on~$V$. Moreover, the bilinear form is bounded on $L^2(\Omega)\times V$. The inf-sup conditions are proved by using the isomorphism
  \[
    \Delta v\in L^2(\Omega),\ v_{|\Gamma}=0\quad \Leftrightarrow \quad v \in V.
  \]
  In particular, we obtain by taking $y=\Delta v$
  \begin{align*}
    \sup_{y\in L^2(\Omega)} \frac{\left|(y,\Delta v)_\Omega\right|}{\|y\|_{L^2(\Omega)}} &\ge
    \frac{(\Delta v,\Delta v)_\Omega}{\|\Delta v\|_{L^2(\Omega)}} =
    \|\Delta v\|_{L^2(\Omega)}=\|v\|_V,
  \end{align*}
  and by taking the solution $v\in V$ of $\Delta v=y$
  \begin{align}\label{infsup2}
    \sup_{v\in V} \frac{\left|(y,\Delta v)_\Omega\right|}{\|v\|_V}  &\ge
    \frac{(y,y)_\Omega}{\|y\|_{L^2(\Omega)}}=\|y\|_{L^2(\Omega)}.
  \end{align}
  The existence of the unique solution $y\in L^2(\Omega)$ of problem
  \eqref{eq:veryweak2} follows from the standard Babu\v ska--Lax--Milgram
  theorem, see for example \cite[Theorem 2.1]{Babuska1970}. The
  a priori estimate follows from \eqref{infsup2}, and
  \eqref{eq:veryweak2} and \eqref{jonny:normalV},
  \begin{align*}
    \|y\|_{L^2(\Omega)} &\le 
    \sup_{v\in V} \frac{\left|(y,\Delta v)_\Omega\right|}{\|v\|_V} =
    \sup_{v\in V} \frac{\left|(u,\partial_n v)_\Gamma - 
    (f,v)_\Omega\right|}{\|v\|_V}
    \le \|u\|_{V_\Gamma'}+\|f\|_{V'}.
  \end{align*}
\end{proof}

Note that if a weak solution $y\in \{v\in H^1(\Omega):v|_\Gamma=u\}$
exists, then with the help of the Green formula
\begin{align}\label{eq:Green}
  (\partial_n v,\chi)_\Gamma &= (\nabla
  v,\nabla\chi)_\Omega + (\Delta v,\chi)_\Omega \quad
  \forall v\in H^1_\Delta(\Omega)\supset V, \ \forall\chi\in
  H^1(\Omega),
\end{align}
see Lemma 3.4 in the paper \cite{Costabel1988} by Costabel, it is also
a very weak solution. (Set $\chi=y$ and use $(\nabla v,\nabla
  y)_\Omega=(f,v)_\Omega$.)

A possible difficulty with this formulation is that a conforming
discretization with a finite-dimensional space $V_h\subset V$ would
require the use of $C^1$-functions. This is simple for one-dimensional
domains $\Omega$ but requires a lot of degrees of freedom in two (and
more) dimensions.

We finish this subsection with a regularity result.
\begin{lemma}\label{lem:regularity}
  The unique solution $y\in L^2(\Omega)$ of problem
  \eqref{eq:veryweak2} with $u\in L^2(\Gamma)$ and $f=0$ belongs to
  $H^{1/2}(\Omega)$ and to
  \begin{align}\label{tildeW12}
    \widetilde W^{1,2}(\Omega):=\{z\in L^2(\Omega): \delta^{1/2}\nabla
    z\in L^2(\Omega)^2\}
  \end{align}
  where $\delta(x)$ is the distance of $x$ to the boundary $\Gamma$.
  Furthermore, there exists a positive constant $c$ such that
  \begin{equation}\label{eq:apriori0}
    \|y\|_{H^{1/2}(\Omega)}+\|\delta^{1/2}\nabla y\|_{L^2(\Omega)^2}
    \leq c\|u\|_{L^2(\Gamma)}.
  \end{equation}
\end{lemma}

\begin{remark}
  The book by Chabrowski, \cite{Chabrowski1991}, deals exclusively with
  the Dirichlet problem with $L^2$ boundary data for elliptic linear
  equations. The solution is searched there in the Sobolev space
  \eqref{tildeW12} but domains of class $C^{1,1}$ were considered
  only.
\end{remark} 

\begin{proof}[Proof of Lemma \ref{lem:regularity}]
  In a first step, we use an integral representation and some
  properties of the layer potentials to get a solution with the
  appropriate regularity.  Theorem 4.2 of Verchota's paper
  \cite{Verchota1984} shows that the operator $\frac12 I+K$, $K$ being
  the boundary double layer potential, is an isomorphism from
  $L^2(\Gamma)$ into itself (see also Corollary 4.5 of
  \cite{MitreaTaylor1999}).  According to the trace property for the
  double layer potential $\cal K$ (see Section 1 and Corollary 3.2 of
  \cite{Verchota1984}), there exists a unique harmonic function $z$
  such that
  \[
    z\to u \hbox{ a.e.  in nontangential cones}
  \]
  with the notation from \cite{Verchota1984}, and is given by
  \begin{equation}\label{eq:intrepresentation}
    z={\cal K} (\tfrac12 I+K)^{-1} u.
  \end{equation}
  But due to the above mentioned isomorphism property and Theorem 1 of
  Costabel's paper \cite{Costabel1988} (and its following Remark), we
  obtain that
  \begin{equation}\label{eq:apriori}
    \|z\|_{H^{1/2}(\Omega)} \leq c\|u\|_{L^2(\Gamma)}.
  \end{equation}
  Note that Theorems 5.3, 5.4 and Corollary 5.5 of the paper 
  \cite{JerisonKenig1995} by Jerison and Kenig also yield
  \begin{equation}\label{eq:apriori2}
    \|\delta^{1/2}\nabla z\|_{L^2(\Omega)^2}
    \leq c\|u\|_{L^2(\Gamma)}.
  \end{equation}

  The second step is to show that $z$ is the very weak solution, hence
  by uniqueness, we will get $y=z$. The regularity of $y$ and the
  estimate \eqref{eq:apriori0} follow from our first step.  For that
  last purpose, we use a density argument.  Indeed let $u_n\in
  H^1(\Gamma)$ be a sequence of functions such that
  \begin{equation}\label{eq:density}
    u_n\to u \hbox{ in } L^2(\Gamma), \hbox{ as }   n\to \infty.
  \end{equation}
  Consider $z_n={\cal K} (\frac12 I+K)^{-1} u_n$ and let $y_n\in
  L^2(\Omega)$ be the unique solution of \eqref{eq:veryweak2} with
  boundary datum $u_n$ and right hand side $f=0$ (that is in
  $H^1(\Omega)$). Then by the estimate \eqref{eq:apriori} and Lemma
  \ref{lem:existenceveryweak}, we get
  \begin{eqnarray*}
    y_n\to y \hbox{ in } L^2(\Omega), \hbox{ as }   n\to \infty,\\
    z_n \to z \hbox{ in } H^{1/2}(\Omega), \hbox{ as }   n\to \infty.
  \end{eqnarray*}
  Furthermore by Theorem 5.15 of \cite{JerisonKenig1995} $z_n$
  satisfies
  \[
    \gamma z_n= u_n \hbox{ on } \Gamma,
  \]
  where $\gamma$ is the trace operator from $H^1(\Omega)$ into
  $H^{1/2}(\Gamma)$.  Hence we directly deduce that $y_n=z_n$ and by
  the above convergence property we conclude that $y=z$.
\end{proof}

\begin{corollary}\label{cor:regularity}
  The unique solution $y\in L^2(\Omega)$ of problem
  \eqref{eq:veryweak2} with $u\in L^2(\Gamma)$ and $f\in
  H^{-1}(\Omega)$ belongs to $H^{1/2}(\Omega)$ and to $\widetilde
  W^{1,2}(\Omega)$ from \eqref{tildeW12}.  There exists a positive
  constant $c$ such that
  \begin{equation*}
    \|y\|_{H^{1/2}(\Omega)}+\|\delta^{1/2}\nabla y\|_{L^2(\Omega)^2}
    \leq c\left(\|u\|_{L^2(\Gamma)}+\|f\|_{H^{-1}(\Omega)}\right).
  \end{equation*}
\end{corollary}

\subsection{Method of transposition employing weighted Sobolev spaces}\label{sec:weighted}
Alternatively to the space $H^1_\Delta(\Omega)\cap H^1_0(\Omega)$, one can use the test space
\[
  V^{2,2}_\beta(\Omega)\cap H^1_0(\Omega), \quad \beta>1-\frac\pi\omega,
\] 
in the non-convex case, where $V^{2,2}_\beta(\Omega)$ is a
weighted Sobolev space of the class
\begin{align}
  \begin{split}\label{def:weightedSobolevSpace}
  V^{k,p}_\beta(\Omega) &:= 
  \left\{v\in\mathcal{D}'(\Omega):\|v\|_{V^{k,p}_\beta(\Omega)}<\infty\right\}, \\
  \|v\|_{V^{k,p}_\beta(\Omega)}^p &:=
  \sum_{|\alpha|\le k}\int_\Omega |r^{\beta-k+|\alpha|}D^\alpha v|^p,
  \end{split}
\end{align}
where we use standard multi-index notation.
For later use we also introduce
\[ 
  L^2_\beta(\Omega):=V^{0,2}_\beta(\Omega).
\]
First derivatives of $V^{2,2}_\beta(\Omega)$-functions belong to $V^{1,2}_{\beta}(\Omega)$ by definition. The trace space of
$V^{1,2}_{\beta}(\Omega)$ is $\prod_{j=1}^N V^{1/2,2}_{\beta}(\Gamma_j)$, see \cite[Lemma 1.2]{mazja:78}
or \cite[Theorem 1.31]{nicaise:93}.
In the next lemma, we will use the spaces
\[
  V_\beta:=\begin{cases}
      V^{2,2}_\beta(\Omega)\cap H^1_0(\Omega) & \text{for }\omega>\pi, \\
      H^2(\Omega)\cap H^1_0(\Omega) & \text{for }\omega<\pi,
    \end{cases} \qquad
  Y_\beta:=\begin{cases}
      L^2_{-\beta}(\Omega) & \text{for }\omega>\pi, \\
      L^2(\Omega) & \text{for }\omega<\pi,
    \end{cases}
\]
for $\beta\in(1-\frac\pi\omega,1]$. We endow $V_\beta$ with the
$V^{2,2}_\beta(\Omega)$-norm for $\omega>\pi$ and the
$H^2(\Omega)$-norm for $\omega<\pi$, as well as $Y_\beta$ with the
$L^2_{-\beta}(\Omega)$-norm for $\omega>\pi$ and the
$L^2(\Omega)$-norm for $\omega<\pi$.

\begin{remark}
  Let us discuss the definition of the space $V_\beta$ and the restriction of the weight $\beta$ to the interval $(1-\frac\pi\omega,1]$ in the non-convex case:
  \begin{enumerate}
    \item We require $\beta\in(1-\frac\pi\omega,1]$ in order to have the isomorphism
    \begin{equation}\label{jonny:Viso}
			\Delta v\in L^2_\beta(\Omega),\ v_{|\Gamma}=0\quad \Leftrightarrow\quad v \in V_\beta.
		\end{equation}
		\item It is possible to use the test space $V^{2,2}_{\beta}(\Omega)\cap H^1_0(\Omega)$ in convex domains as well. However, this implies a loss of information about the solution since this test space is smaller than $H^2(\Omega)\cap H^1_0(\Omega)$ due to the fact that the weight $\beta$ can be negative.
		\item Note that $V^{1/2,2}_{\beta}(\Gamma_j)\hookrightarrow L^2(\Gamma_j)$ for $\beta \leq \frac12$. This implies $L^2(\Gamma)=\left(\prod_{j=1}^N L^2(\Gamma_j)\right)'\hookrightarrow \left(\prod_{j=1}^N V^{1/2,2}_{\beta}(\Gamma_j)\right)'$.
		This means that $L^2$-boundary data are included in the following discussion if $\beta\leq \frac12$.
  \end{enumerate}
\end{remark}

\begin{lemma}\label{lem:existence_using_weighted_spaces}
  Let $\beta\in(1-\frac\pi\omega,1]$
  and assume that $f\in V_\beta'$ and $u\in \left(\prod_{j=1}^N
  V^{1/2,2}_{\beta}(\Gamma_j)\right)'$. Then there
  exists a unique solution 
  \begin{align}\label{eq:veryweakbeta}
    y\in Y_\beta\hookrightarrow L^2(\Omega):\quad (y, \Delta
    v)_\Omega = (u,\partial_n v)_\Gamma -
    (f,v)_\Omega \quad\forall v\in V_\beta
  \end{align}
  with
  \[
    \|y\|_{Y_\beta}\leq c\left(\|f\|_{V_\beta'}+\|u\|_{\left(\prod_{j=1}^N
  V^{1/2,2}_{\beta}(\Gamma_j)\right)'}\right).
  \]
\end{lemma}

\begin{proof}
  The convex case was already treated in Lemma \ref{lem:MRV}, hence we
  focus on the non-convex case. We proceed as in Lemma \ref{lem:existenceveryweak}. The right hand
  side of \eqref{eq:veryweakbeta} defines a continuous functional on $V_\beta$.
  Furthermore, the bilinear form is bounded on $Y_\beta\times V_\beta$,
  \[ 
    \left|(y, \Delta v)_\Omega\right| \le \|r^{-\beta} y\|_{L^2(\Omega)} \,
    \|r^\beta \Delta v\|_{L^2(\Omega)} \le \|y\|_{Y_\beta}\,\|v\|_{V_\beta}.
  \]
  The inf-sup conditions for the bilinear form are proved  by using the isomorphism \eqref{jonny:Viso},
  in particular $\|v\|_{V_\beta}\le c\|\Delta v\|_{L^2_\beta(\Omega)}$;
  we obtain by taking $y=r^{2\beta}\Delta v$
  \begin{align*}
    \sup_{y\in Y_\beta} \frac{\left|(y,\Delta v)_\Omega\right|}{\|y\|_{Y_\beta}} &\ge
    \frac{(r^{\beta}\Delta v,r^{\beta}\Delta v)_\Omega}%
         {\|r^{2\beta}\Delta v\|_{L^2_{-\beta}(\Omega)}} =
    \frac{\|\Delta v\|_{L^2_\beta(\Omega)}^2}%
         {\|\Delta v\|_{L^2_{\beta}(\Omega)}} \ge c^{-1}\|v\|_{V_\beta},
  \end{align*}
  and by taking the solution $v\in V_\beta$ of $\Delta v=r^{-2\beta}y$ with
  $\|v\|_{V_\beta}\le c\|\Delta v\|_{L^2_\beta(\Omega)} =
  c\|y\|_{L^2_{-\beta}(\Omega)}$
  \begin{align}
    \sup_{v\in V_\beta} \frac{\left|(y,\Delta v)_\Omega\right|}{\|v\|_{V_\beta}}  &\ge
    \frac{(r^{-\beta}y,r^{-\beta}y)_\Omega}{c\|y\|_{L^2_{-\beta}(\Omega)}}=
    c^{-1}\|y\|_{L^2_{-\beta}(\Omega)}.\label{jonny:Vsup2}
  \end{align}
  The existence of the unique solution $y\in L^2_{-\beta}(\Omega)$ of problem
  \eqref{eq:veryweakbeta} follows now from the standard Babu\v ska--Lax--Milgram
  theorem, see for example \cite[Theorem 2.1]{Babuska1970}. The a~priori estimate is obtained with \eqref{jonny:Vsup2} and
  \eqref{eq:veryweakbeta},
  \[
    \|y\|_{Y_\beta}\leq c\sup_{v\in V_\beta} \frac{\left|(y,\Delta v)_\Omega\right|}{\|v\|_{V_\beta}}=c\sup_{v\in V_\beta} \frac{\left| (u,\partial_n v)_\Gamma -
    (f,v)_\Omega\right|}{\|v\|_{V_\beta}}.
  \]
	This ends the proof since we already noticed that the enumerator defines a continuous functional on $V_\beta$.
\end{proof}

\subsection{\label{sec:Berggren}Berggren's approach}
Berggren's approach \cite{Berggren2004} avoids test functions in $H^1_\Delta(\Omega)\cap H^1_0(\Omega)$ in an explicit way.
It can be explained as if we substitute $\varphi:=-\Delta v$ and
$\normaltrace:=\partial_n v$ in \eqref{eq:veryweak2},
\begin{align}\label{eq:berggren1}
  y\in L^2(\Omega): \quad (y, \varphi)_\Omega &=
  -(u,\normaltrace)_\Gamma + (f,v)_\Omega \quad
  \forall\varphi\in L^2(\Omega).
\end{align}
The relationship between $\varphi\in L^2(\Omega)$ and both
$v\in V$ and $\normaltrace \in V_\Gamma$ can be expressed by the weak formulation of the Poisson
equation,
\begin{align}\label{eq:berggren2}
  v\in V: \quad (\nabla v,\nabla\psi)_\Omega &=
  (\varphi,\psi)_\Omega \quad \forall\psi\in H^1_0(\Omega)
\end{align}
and a reformulation of the Green formula \eqref{eq:Green}
in the form
\begin{align}\label{eq:berggren3}
  \normaltrace\in V_\Gamma: \quad (\normaltrace,\chi)_\Gamma &=
  (\nabla v,\nabla\chi)_\Omega - (\varphi,\chi)_\Omega
  \quad \forall\chi\in H^1(\Omega)\setminus H^1_0(\Omega).
\end{align}

Note that Berggren's formulation is
not a system with three unknown functions since the second and third
equations compute actions on the test function $\varphi$. Indeed, let $S:L^2(\Omega)\rightarrow V$ and $F:L^2(\Omega)\rightarrow V_\Gamma$ be the solution operators of \eqref{eq:berggren2} and \eqref{eq:berggren3}, respectively, defined by $S\varphi:=v$ and $F\varphi:=\normaltrace$, then we could
also write
\begin{align*}
  y\in L^2(\Omega): \quad (y, \varphi)_\Omega =
  -(u,F\varphi)_\Gamma - (f,S\varphi)_\Omega \quad
  \forall\varphi\in L^2(\Omega)
\end{align*}
instead of \eqref{eq:berggren1}.

\begin{lemma}\label{jonny:berg}
  Berggren's formulation \eqref{eq:berggren1}, \eqref{eq:berggren2},
  \eqref{eq:berggren3} is equivalent to the formulation
  \eqref{eq:veryweak2}.
\end{lemma}

\begin{proof}
  We first assume that $y\in L^2(\Omega)$ satisfies
  \eqref{eq:veryweak2} and show
  \eqref{eq:berggren1}--\eqref{eq:berggren3}. For any $\varphi\in
  L^2(\Omega)$ let $v$ be the variational solution of
  $-\Delta v=\varphi$ defined by \eqref{eq:berggren2}, hence $v\in V$ and
  $\normaltrace:=\partial_n v \in V_\Gamma$. Based on the formula
  \eqref{eq:Green}, we obtain \eqref{eq:berggren3}. With
  \eqref{eq:veryweak2} we finally get also \eqref{eq:berggren1}.

  Let now $y$ satisfy \eqref{eq:berggren1}--\eqref{eq:berggren3}.
  Since $\varphi\in L^2(\Omega)$ we get from \eqref{eq:berggren2} that
  $v\in V$. Moreover, we obtain $\normaltrace=\partial_n v\in
  V_\Gamma$ from \eqref{eq:berggren3}. Hence equation
  \eqref{eq:berggren1} becomes
  \[
    -(y, \Delta v)_\Omega = -(u,\partial_n v)_\Gamma +
    (f,v)_\Omega 
    \quad\forall v\in V
  \]
  due to the isometry between $L^2(\Omega)$ and $V$.
\end{proof}
\begin{remark}
  Berggren used the regularity $v\in H^{3/2+\epsilon}(\Omega)$ with some $\epsilon>0$ which implies $\normaltrace\in H^{\epsilon}(\Gamma)$. However, he did not consider the maximal domain of the elliptic operator, i.e., $v \in V$ and $\normaltrace\in V_\Gamma$. But with the explanations of Subsection \ref{sec:general} these regularities should be obvious. Thus the result of Lemma \ref{jonny:berg} is slightly more general than that of Berggren.
\end{remark}

\subsection{\label{sec:regularizationstrategy}The regularization approach}

A further idea is to regularize the boundary data and then to apply
standard methods. This approach has already been considered within the proof 
of Lemma~\ref{lem:MRV}. In contrast, we do not use the isomorphism
\[
  \Delta w\in L^2(\Omega),\ w_{|\Gamma}=0\quad\Leftrightarrow\quad w\in H^2(\Omega)\cap H^1_0(\Omega),
\]
which can only be employed in case of convex domains, but the isomorphism
\[
	\Delta v\in L^2(\Omega),\ v_{|\Gamma}=0\quad \Leftrightarrow \quad v \in V.
\]
This allows us to apply the regularization approach in the non-convex case as well.
Moreover, we propose two different strategies how the regularized Dirichlet boundary
data can be constructed in an explicit way. Thereby we will be able in Subsection \ref{sec:numreg}
to calculate approximate solutions of the regularized problems based on a finite element method.
For the data we assume henceforth $u\in L^2(\Gamma)$ and $f\in H^{-1}(\Omega)$. This is not only for simplicity but
also due to the fact that already for Dirichlet boundary data in $L^2(\Gamma)$ the convergence rates of the approximate
solutions in Subsection \ref{sec:numreg} tend to zero as the maximal interior angle tends to $2\pi$.

We start with general convergence results for the regularized solutions.
To this end let $u^h\in H^{1/2}(\Gamma)$ be a sequence of functions such that
\begin{align*}
  \lim_{h\to0}\|u-u^h\|_{L^2(\Gamma)}=0.
\end{align*}
Let now $y^h\in Y_*^h:=\{v\in H^1(\Omega): v|_\Gamma=u^h\}$ be the
variational solution,
\begin{align}\label{eq:regsol}
  y^h\in Y_*^h: \quad (\nabla y^h,\nabla v)_\Omega =
  (f,v)_\Omega \quad\forall v\in H^1_0(\Omega).
\end{align}

\begin{lemma}\label{lem:exregsol}
  Let $u\in L^2(\Gamma)$ and $f\in H^{-1}(\Omega)$. Then the limit $y:=\lim\limits_{h\to0}y^h$ exists, belongs to $L^2(\Omega)$, and is
  the very weak solution, that means it satisfies \eqref{eq:veryweak2}.
\end{lemma}

\begin{proof}
  First we show that $y^h$ is a Cauchy sequence in $L^2(\Omega)$.
  From (\ref{eq:regsol}) and Green's formula, we have for any $v\in
  V$,
  \begin{align*}
    (f,v)_\Omega = (\nabla y^h,\nabla v)_\Omega &=
    -( y^h,\Delta v)_\Omega + (y^h, \partial_n v)_\Gamma, \\
    (f,v)_\Omega = (\nabla y^{h'},\nabla v)_\Omega &=
    -( y^{h'},\Delta v)_\Omega + (y^{h'}, \partial_n v)_\Gamma.
  \end{align*}
  Hence due to $y^h=u^h$ and $y^{h'}=u^{h'}$ on $\Gamma$, we
  deduce that
  \begin{align}\label{eq:serge1a}
    ( y^h-y^{h'},\Delta v)_\Omega=(u^h-u^{h'}, \partial_n v)_\Gamma 
    \quad\forall v\in V.
  \end{align}
  Now for any $z\in L^2(\Omega)$, let $v_z\in V$ be such that
  \begin{align}\label{eq:sergepbDir}
    \Delta v_z=z,
  \end{align}
  that clearly satisfies
  \begin{align}\label{eq:serge2a}
    \|\partial_n v_z\|_{L^2(\Gamma)} \le 
    c\|v_z\|_{H^s(\Omega)}\le c \|z\|_{L^2(\Omega)}
  \end{align}
  with some $s\in\left(\frac32,1+\lambda\right)$, $s\le 2$. Finally, we obtain
  with  \eqref{eq:serge1a} and \eqref{eq:serge2a}
  \begin{align*}
    \|y^h-y^{h'}\|_{L^2(\Omega)} &=  \sup_{z\in L^2(\Omega), z\ne 0}
    \frac{(y^h-y^{h'}, z)_\Omega}{\|z\|_{L^2(\Omega)}} =
    \sup_{z\in L^2(\Omega), z\ne 0}
    \frac{(u^h-u^{h'}, \partial_n v_z)_\Gamma}{\|z\|_{L^2(\Omega)}} \\ &\le
    \|u^h-u^{h'}\|_{L^2(\Gamma)} \sup_{z\in L^2(\Omega), z\ne 0}
    \frac{\|\partial_n v_z\|_{L^2(\Gamma)}}{\|z\|_{L^2(\Omega)}} =
    c \|u^h-u^{h'}\|_{L^2(\Gamma)}.
  \end{align*}
  Since $u^h$ converges in $L^2(\Gamma)$, it is a Cauchy sequence and
  hence also $y^h$ is a Cauchy sequence and converges in $L^2(\Omega)$
  by the completeness of $L^2(\Omega)$.

  From $V\subset H^1_0(\Omega)$ we obtain by \eqref{eq:regsol} and
  the Green formula \eqref{eq:Green}
  \begin{align*}
    (f,v)_\Omega &= (\nabla y^h,\nabla v)_\Omega = 
    (\Delta v,y^h)_\Omega - (\partial_n v,y^h)_\Gamma \\ &= 
    (\Delta v,y^h)_\Omega - (\partial_n v,u^h)_\Gamma
    \quad\forall v\in V.     
  \end{align*}
  Since $\Delta v\in L^2(\Omega)$ and $\partial_n v\in L^2(\Gamma)$ we
  can pass to the limit and obtain that the limit function $y$
  satisfies \eqref{eq:veryweak2}.
\end{proof}

We can estimate the regularization error by a similar technique.
\begin{lemma}\label{lem:regularizationerror}
  Let $s=\frac12$ if $\Omega$ is convex and $s\in[0,\lambda-\frac12)$
  if $\Omega$ is non-convex.  Then the estimate
  \begin{align*}
    \|y-y^h\|_{L^2(\Omega)}\le c \|u-u^h\|_{H^{-s}(\Gamma)}
  \end{align*}
  holds.
\end{lemma}

\begin{proof}
  We use the approach of the proof of Lemma \ref{lem:exregsol} and just
  replace $u^{h'}$ by $u$ and $y^{h'}$ by $y$ to get
  \begin{align}\label{eq:serge1}
    ( y-y^h,\Delta v)_\Omega=(u-u^h, \partial_n v)_\Gamma\quad 
    \forall v\in V.
  \end{align}
  Again, for any $z\in L^2(\Omega)$, we let $v_z\in V$ be such that
  $\Delta v_z=z$ but estimate now in a sharper way
  \begin{align}\label{eq:serge2}
    \|\partial_n v_z\|_{H^s(\Gamma)} \le 
    c\|v_z\|_{H^{s+3/2}(\Omega)}\le c \|z\|_{L^2(\Omega)}
  \end{align}
  As in the previous proof we get
  \begin{align*}
    \|y-y^h\|_{L^2(\Omega)}&=  \sup_{z\in L^2(\Omega), z\ne 0}
    \frac{(u-u^h,\partial_n v_z)_\Omega}{\|z\|_{L^2(\Omega)}} \\ &\leq
    \|u-u^h\|_{H^{-s}(\Gamma)} \sup_{z\in L^2(\Omega), z\ne 0}
    \frac{\|\partial_n v_z\|_{H^s(\Gamma)}}{\|z\|_{L^2(\Omega)}} \\&\le c
    \|u-u^h\|_{H^{-s}(\Gamma)}.
  \end{align*}
  Actually, the proof is for $s>0$ but of course the statement
  holds when $s$ is decreased.
\end{proof}

A choice for the construction of the regularized function $u^h$ could
be the use of the $L^2(\Gamma)$-projection $\Pi_hu$ into a piecewise
polynomial space on the boundary (which we call $Y_h^\partial$ in
Section \ref{sect:num}) or the use of the Carstensen interpolant
$C_hu$, see \cite{Carstensen1999}. Namely, if ${\mathcal N}_{\Gamma}$
is the set of nodes of the triangulation on the boundary, we set
\[
  C_hu=\sum_{x\in {\mathcal N}_{\Gamma}} \pi_x(u)\lambda_x,
\]
where $\lambda_x$ is the standard hat function related to $x$ and
\[
  \pi_x(u)=\frac{\int_{\Gamma} u\lambda_x} {\int_{\Gamma} \lambda_x}
=\frac{(u,\lambda_x)_\Gamma}{(1,\lambda_x)_\Gamma}.
\]
The advantages of the interpolant in comparison with the
$L^2$-projection are its local definition and the property
\[ u\in[a,b]\quad\Rightarrow\quad C_hu\in[a,b],\] 
see \cite{ReyesMeyerVexler2008};
a disadvantage for our application in optimal control is that
$C_hu_h\not=u_h$ for piecewise linear $u_h$.
We prove now regularization error estimates for the case that the
regularized function $u^h$ is constructed via $C_hu$ or $\Pi_hu$.

\begin{lemma}\label{lem:regularizationerrorforu}
  If $u^h$ is the piecewise linear Carstensen interpolant of $u$
  or the $L^2(\Gamma)$-projection of $u$ into a space of piecewise
  linear functions, then there holds
  \[
    \|u-u^h\|_{H^{-s}(\Gamma)}\le ch^s\|u\|_{L^2(\Gamma)},
    \quad s\in[0,1]
  \]
  as well as
  \[
    \|u^h\|_{L^2(\Gamma)}\le c\|u\|_{L^2(\Gamma)}.
  \]
\end{lemma}

\begin{proof}
  The interpolation error estimate and stability result are derived in
  \cite{Carstensen1999,ReyesMeyerVexler2008} for domains. The proofs
  can be transferred to estimates on the boundary $\Gamma$. For the
  sake of completeness we sketch these proofs in the Appendix, see
  Lemma \ref{lem:carstensen}.

  In the case of the $L^2$-projection, the second estimate holds with
  constant one.  For the first estimate, we notice that by using the
  properties of the $L^2$-projection, we have
  \begin{align*}
    (u-\Pi_hu, \varphi)_\Gamma &= 
    (u-\Pi_hu, \varphi-\Pi_h\varphi)_\Gamma =
    (u, \varphi-\Pi_h\varphi)_\Gamma \le
    \|u\|_{L^2(\Gamma)}\|\varphi-\Pi_h\varphi\|_{L^2(\Gamma)}\\ & \le 
    \|u\|_{L^2(\Gamma)}\|\varphi-C_h\varphi\|_{L^2(\Gamma)} \le 
    ch^s\|u\|_{L^2(\Gamma)}\|\varphi\|_{H^{s}(\Gamma)}
  \end{align*}
  where we used Lemma \ref{lem:carstensen} in the last step.
  We conclude 
  \[
    \|u-u^h\|_{H^{-s}(\Gamma)}=
    \sup_{\varphi\in H^{s}(\Gamma), \varphi\ne 0}
    \frac{(u-\Pi_hu, \varphi)_\Gamma}{\|\varphi\|_{H^{s}(\Gamma)}} \le
    ch^s\|u\|_{L^2(\Gamma)}.
  \]
  which is the assertion.
\end{proof}

By setting $s=0$ in the previous lemma we obtain \[\|u-u^h\|_{L^2(\Gamma)} \leq c \|u\|_{L^2(\Gamma)}\] for the different choices of the regularized function $u^h$. This means that for $u\in L^2(\Gamma)$ the
difference $u-u^h$ is uniformly bounded in $L^2(\Gamma)$ independent of $h$.
However, we require strong convergence in $L^2(\Gamma)$ for $u\in L^2(\Gamma)$, i.e. \[\lim_{h\to0}\|u-u^h\|_{L^2(\Gamma)}=0.\]
This is subject of the next lemma. A comparable result for the Ritz-projection can be found in e.g. \cite[Theorem 3.2.3]{ciarlet:78}.
\begin{lemma}
 Let $u\in L^2(\Gamma)$ and let $u^h$ be the piecewise linear Carstensen interpolant of $u$
  or the $L^2(\Gamma)$-projection of $u$ into a space of piecewise
  linear functions. Then there holds
  \[
   \lim_{h\to0}\|u-u^h\|_{L^2(\Gamma)}=0.
  \]
\end{lemma}
\begin{proof}
We show the validity of this lemma for the Carstensen interpolant. The convergence result for the $L^2(\Gamma)$-projection can be proven analogously.

Due to the compact embedding $H^{1}(\Gamma) \stackrel{c}{\hookrightarrow} L^2(\Gamma)$ there exists a sequence of functions $u_n\in H^{1}(\Gamma)$ such that
\[
 \lim_{n\to\infty}\|u-u_n\|_{L^2(\Gamma)}=0
\]
with
\begin{equation}\label{jonny:c_n}
 \|u_n\|_{H^{1}(\Gamma)}\leq c_n,
\end{equation}
where the constant $c_n$ may depend on $n$. Thus, for every $\varepsilon>0$ there is a positive integer $N$ such that
\begin{equation}\label{jonny:varepsilon2}
 \|u-u_n\|_{L^2(\Gamma)}\leq \frac\varepsilon2
\end{equation}
for $n\geq N$. By inserting the function $u_n$ and its Carstensen interpolant $C_hu_n$ as intermediate functions into the desired term we obtain
\[
 \|u-C_hu\|_{L^2(\Gamma)}\leq \|u-u_n\|_{L^2(\Gamma)}+\|u_n-C_hu_n\|_{L^2(\Gamma)}+\|C_h(u_n-u)\|_{L^2(\Gamma)}.
\]
The results of Lemma \ref{lem:regularizationerrorforu} and the inequalities \eqref{jonny:c_n} and \eqref{jonny:varepsilon2} imply
\[
 \|u-C_hu\|_{L^2(\Gamma)}\leq c\left(\frac\varepsilon2 + h c_n\right).
\]
Since for every constant $c_n$ there is a parameter $h_n$ such that
$hc_n\leq \frac\varepsilon2$ for all $h\leq h_n$ we arrive at
\[
 \|u-C_hu\|_{L^2(\Gamma)}\leq c\varepsilon
\]
and the desired result follows.
\end{proof}

\subsection{A negative result}\label{sec:negative}

Since the boundary datum $u\in L^2(\Gamma)\hookrightarrow L^1(\Gamma)$
is trace of a function $w\in W^{1,1}(\Omega)$ an idea could be to
search $y=y_0+w$ with $y_0|_\Gamma=0$ and
\begin{align*}
  y_0\in W^{1,1}_0(\Omega):\quad (\nabla y_0,\nabla v)_\Omega =
  (f,v)_\Omega- (\nabla w,\nabla v)_\Omega =: ( F,v)_\Omega \quad
  \forall v\in W^{1,\infty}_0(\Omega).
\end{align*}
However, bilinear forms $a(u,v):W^{1,p}_0(\Omega)\times
W^{1,p'}_0(\Omega)\to\mathbb{C}$
are investigated in \cite{Simader1972,SimaderSohr1996} and it is shown that
\begin{align*}
  \forall F\in W^{1,p'}(\Omega)^*\quad \exists z\in W^{1,p}_0(\Omega):\quad
  a(z,v)=( F,v)_\Omega \quad \forall v\in W^{1,p'}_0(\Omega)
\end{align*}
holds for $p\in(1,\infty)$ only since $c_p\to\infty$ for $p\to1$ in
the inf-sup condition
\begin{align*}
  c_p\sup\limits_{\phi\in W^{1,p'}_0(\Omega)} \frac{(\nabla
    z,\nabla\phi)_\Omega}{\|\nabla\phi\|_{L^{p'}(\Omega)}} \ge \|\nabla
  z\|_{L^{p}(\Omega)} \quad\text{for }1<p<\infty.
\end{align*}

\section{\label{sect:num}Discretization of the boundary value problem}

Let $\left(\mathcal{T}_h\right)_{h>0}$ be a family of conforming, quasi-uniform finite
element meshes, and introduce the finite element spaces
\begin{align}\label{eq:discretespaces}
  Y_h = \{v_h\in H^1(\Omega): v_h|_T\in\mathcal{P}_1\ \forall
  T\in\mathcal{T}_h\}, \quad Y_{0h} = Y_h\cap H^1_0(\Omega),\quad
  Y_h^\partial = Y_h|_{\partial\Omega}.
\end{align}

\subsection{Berggren's approach}\label{sec:numberg}

Let $u\in L^2(\Gamma)$ and $f\in H^{-1}(\Omega)$. We discretize the formulation \eqref{eq:berggren1},
\eqref{eq:berggren2}, \eqref{eq:berggren3} in a straightforward
manner,
\begin{align*}
  &y_h\in Y_h: & (y_h, \varphi_h)_\Omega &=
  -(u,\normaltrace_h)_\Gamma + (f,v_h)_\Omega &&
  \forall\varphi_h\in Y_h, \\
  &v_h\in Y_{0h}: &\hspace{-0.5em} (\nabla
  v_h,\nabla\psi_h)_\Omega &= (\varphi_h,\psi_h)_\Omega
  &&
  \forall\psi_h\in Y_{0h}, \\
  &\normaltrace_h\in Y_h^\partial: & (\normaltrace_h,\chi_h)_\Gamma &=
  (\nabla v_h,\nabla\chi_h)_\Omega -
  (\varphi_h,\chi_h)_\Omega && \forall\chi_h\in Y_h\setminus
  Y_{0h}.
\end{align*}
Note that $\normaltrace_h\neq F\varphi_h$ and $v_h\neq S\varphi_h$ with $F$ and $S$ from Subsection \ref{sec:Berggren}, hence it is
only an approximate Galerkin formulation.

Berggren showed in Theorem 5.2 of \cite{Berggren2004} that this formulation
is equivalent with the standard finite element approximation with
$L^2$-projection of the boundary data, 
\begin{align}\label{eq:dicretizationBerggren}
  &y_h\in Y_h: & (\nabla y_h,\nabla\varphi_h)_\Omega &=
  (f,\varphi_h)_\Omega &&  \forall\varphi_h\in Y_{0h}, \\
  && (y_h,\varphi_h)_\Gamma &= (u,\varphi_h)_\Gamma &&
  \forall\varphi_h\in Y_h^\partial.
\end{align}
Note that $y_h=u_h$ on $\Gamma$ when $u=u_h\in Y_h^\partial$. This
will be of interest in the discretization of optimal control
problems.

Berggren proved also the following discretization error estimate, see
Theorem 5.5 of~\cite{Berggren2004}.

\begin{lemma}
  Let $\omega$ be the maximal interior angle of the domain $\Omega$,
  and denote by $\lambda=\pi/\omega$ the corresponding singularity
  exponent.  Let $s'\in(0,\frac12]$ be a real number with
  $s'<\lambda-\frac12$, and let $s\in[0,s')$ be a further real number.
  Then the error estimate
  \begin{align*}
    \|y-y_h\|_{L^2(\Omega)}\le c\left( h^s\|y\|_{H^s(\Omega)} +
    h^{s'}\|u\|_{L^2(\Gamma)} + h^{s'+1/2}\|f\|_{H^{-1}(\Omega)}\right) =
    \mathcal{O}(h^s)
  \end{align*}
  holds, this means that we have convergence order $s$,
  \[
    s=\min\left\{\tfrac12,\lambda-\tfrac12\right\}-\varepsilon=
    \begin{cases}
      \tfrac12-\varepsilon & \text{for convex domains,} \\
      \lambda-\tfrac12-\varepsilon & \text{for non-convex domains,}
    \end{cases}
  \]
  $\varepsilon>0$ arbitrary.
\end{lemma}
Note that $s\to0$ for $\omega\to2\pi$.  

We will show in the next section that in the convex case the
convergence order is $\frac12$, without $\varepsilon$.

\subsection{Regularization approach}\label{sec:numreg}

We consider a regularization strategy such that $u^h\in Y_h^\partial$,
see Subsection \ref{sec:regularizationstrategy}. Recall that the
corresponding solution $y^h\in Y_*^h:=\{v\in H^1(\Omega):
v|_\Gamma=u^h\}$ is defined via \eqref{eq:regsol}. For a
regularization using the $L^2(\Gamma)$-projection or the Carstensen
interpolant we have the regularization error estimate
\begin{align}\label{reg_err_est}
  \|y-y^h\|_{L^2(\Omega)} &\le c\|u-u^h\|_{H^{-s}(\Gamma)} \le c
  h^s \|u\|_{L^2(\Gamma)}
\end{align}
with $s=\frac12$ if $\Omega$ is convex and
$s\in[0,\lambda-\frac12)$ if $\Omega$ is non-convex, see  Lemmas
  \ref{lem:regularizationerror} and \ref{lem:regularizationerrorforu}.\pagebreak[3]

The finite element solution $y_h$ is now searched in $Y_{*h}:=
Y_*^h\cap Y_h$ and is defined in the classical way,
\begin{align}\label{eq:serge20/06:6}
  y_h\in Y_{*h}:\quad (\nabla y_h,\nabla v_h)_\Omega =
  (f,v_h)_\Omega\quad\forall v_h\in Y_{0h}.
\end{align}
Note that, if we construct $u^h$ by the $L^2$-projection, we recover
the Berggren approach as a special case.

\begin{lemma}\label{lem:discrerrorestconvex}
  The finite element error estimate 
  \begin{align*}
    \|y^h-y_h\|_{L^2(\Omega)}\le ch^s   
    \left(h^{1/2}\|f\|_{H^{-1}(\Omega)}+\|u\|_{L^2(\Gamma)}\right)
  \end{align*}
  holds for $s=\frac12$ in the convex case and for $s<\lambda-\frac12$
  in the non-convex case.
\end{lemma}

Before we prove this lemma we can immediately imply the following
final error estimate for this approach. By the triangle inequality we have
  \begin{align*}
    \|y-y_h\|_{L^2(\Omega)} &\le 
    \|y-y^h\|_{L^2(\Omega)} + \|y^h-y_h\|_{L^2(\Omega)}. 
  \end{align*}
  The first term is already treated in \eqref{reg_err_est}.
 The second term is treated in Lemma \ref{lem:discrerrorestconvex}. 

\begin{corollary}\label{lem:globalerrorestconvex}
  The discretization error estimate
  \begin{align*}
    \|y-y_h\|_{L^2(\Omega)}\le ch^s   
    \left(h^{1/2}\|f\|_{H^{-1}(\Omega)}+\|u\|_{L^2(\Gamma)}\right)
  \end{align*}
  holds for $s=\frac12$ in the convex case and
  $s\in[0,\lambda-\frac12)$ in the non-convex case. 
\end{corollary}

Note that the order can be improved if the boundary datum $u$ is more
regular, see Remark \ref{rem:5.3}.

\begin{proof}[Proof of Lemma \ref{lem:discrerrorestconvex}]
  Let $B_hu^h\in Y_{*h}$ be the discrete harmonic extension defined by
  \[
    (\nabla B_hu^h,\nabla v_h)_\Omega=0 \quad\forall v_h\in Y_{0h}
  \]
  which satisfies
  \begin{align}\label{eq:serge3}
    \|\nabla B_h u^h\|_{L^2(\Omega)}\le 
    c  \|u^h\|_{H^{1/2}(\Gamma)} \le
    c h^{-1/2} \|u^h\|_{L^2(\Gamma)} \le 
    c h^{-1/2} \|u\|_{L^2(\Gamma)}.
  \end{align}
  The first estimate can be cited from \cite[Lemma
  3.2]{MayRannacherVexler2008}, the second follows from an inverse
  inequality, the third from Lemma \ref{lem:regularizationerrorforu}.

  Now we notice that
  \[
    y^h=B_hu^h+y_0^h \quad \text{as well as}\quad y_h=B_hu^h+y_{0h},
  \]
  where $y_0^h\in H^1_0(\Omega)$ and $y_{0h} \in Y_{0h}$ satisfy 
  \begin{align}\label{eq:serge5}
    (\nabla y_0^h,\nabla v)_\Omega =
    (f,v)_\Omega-(\nabla (B_hu^h),\nabla v)_\Omega
    \quad\forall v\in H^1_0(\Omega), \\\label{eq:serge6}
    (\nabla y_{0h},\nabla v_h)_\Omega =
    (f,v_h)_\Omega-(\nabla (B_hu^h),\nabla
    v_h)_\Omega\quad\forall v_h\in Y_{0h}.
  \end{align}
  Hence $y_{0h}$ is the Galerkin approximation of $y_0^h$ and
  therefore (since $Y_{0h}\subset H^1_0(\Omega)$)
  \begin{align}\label{eq:sergeGalerkin}
    (\nabla (y_0^h-y_{0h}),\nabla v_h)_\Omega =0\quad\forall
    v_h\in Y_{0h}.
  \end{align}

  By taking $v=y_0^h$ in (\ref{eq:serge5}) (resp. $v_h=y_{0h}$ in
  (\ref{eq:serge6})), we see that
  \begin{eqnarray*}
    \|\nabla y_0^h\|^2_{L^2(\Omega)} \le  \|f\|_{H^{-1}(\Omega)}  \|y_0^h\|_{H^1(\Omega)}
    +\|\nabla (B_hu^h)\|_{L^2(\Omega)} \|\nabla y_0^h\|_{L^2(\Omega)},\\
    \|\nabla y_{0h}\|^2_{L^2(\Omega)} \le  \|f\|_{H^{-1}(\Omega)}  \|y_{0h}\|_{H^1(\Omega)}
    +\|\nabla (B_hu^h)\|_{L^2(\Omega)} \|\nabla y_{0h}\|_{L^2(\Omega)}.
  \end{eqnarray*}
  By the Poincar\'e inequality we obtain
  \begin{eqnarray*}
    \|\nabla y_0^h\|_{L^2(\Omega)} \le c
    \left( \|f\|_{H^{-1}(\Omega)}+\|\nabla (B_hu^h)\|_{L^2(\Omega)}\right),\\
    \|\nabla y_{0h}\|_{L^2(\Omega)} \le c\left(  \|f\|_{H^{-1}(\Omega)} 
      +\|\nabla (B_hu^h)\|_{L^2(\Omega)}\right).
  \end{eqnarray*}
  With the help of (\ref{eq:serge3}) we arrive at
  \begin{align}\label{eq:serge7}
    \|\nabla y_0^h\|_{L^2(\Omega)}+\|\nabla y_{0h}\|_{L^2(\Omega)} \le
    c\left( \|f\|_{H^{-1}(\Omega)}+h^{-1/2}\| u\|_{L^2(\Gamma)}\right).
  \end{align}

  Now as before we start with
  \[
    \|y_0^h-y_{0h}\|_{L^2(\Omega)}=\sup_{z\in L^2(\Omega), z\ne
    0}\frac{(y_0^h-y_{0h}, z)_\Omega}{\|z\|_{L^2(\Omega)}}.
  \]
  Letting again $v_z\in V$ be such that $\Delta v_z=z$, we get
  \[
    \|y_0^h-y_{0h}\|_{L^2(\Omega)}=\sup_{z\in L^2(\Omega), z\ne 0}
    \frac{(\nabla(y_0^h-y_{0h}), \nabla v_z)_\Omega}{\|z\|_{L^2(\Omega)}},
  \]
  and therefore thanks to (\ref{eq:sergeGalerkin}) we arrive at
  \[
  \|y_0^h-y_{0h}\|_{L^2(\Omega)}=\sup_{z\in L^2(\Omega), z\ne 0}
   \frac{(\nabla(y_0^h-y_{0h}), \nabla (v_z-I_h v_z))_\Omega}%
   {\|z\|_{L^2(\Omega)}},
  \]
  where $I_h$ is the standard Lagrange interpolation operator. By the
  Cauchy-Schwarz inequality, the well-known estimate
  \[
    \|v_z-I_h v_z\|_{1,\Omega}\le c h^{1/2+s} \|v_z\|_{H^{3/2+s}(\Omega)}
    \le c h^{1/2+s} \|z\|_{L^2(\Omega)},
  \]
  with $s=\frac12$ in the convex case and $s<\lambda-\frac12$ in the non-convex
  case, and the estimate (\ref{eq:serge2}) we obtain
  \[
    \|y_0^h-y_{0h}\|_{L^2(\Omega)}\le c h^{1/2+s}
    \|\nabla(y_0^h-y_{0h})\|_{L^2(\Omega)}.
  \]

  Using the a priori estimate (\ref{eq:serge7}), we conclude that
  \[
    \|y_0^h-y_{0h}\|_{L^2(\Omega)}\le ch^s 
    \left(h^{1/2}\|f\|_{H^{-1}(\Omega)}+\| u\|_{L^2(\Gamma)}\right).
  \]
\end{proof}
Note that we have now proved the convergence order $\frac12$ for
Berggren's approach, in the convex case.

\pagebreak[3]

\section{\label{sect:test}Numerical test}

This section is devoted to the numerical verification of our
theoretical results. For that purpose we present an example with known
solution. Furthermore, to examine the influence of the corner
singularities, we consider different polygonal domains $\Omega_\omega$
depending on an interior angle $\omega\in(0,2\pi)$. These domains are defined by
\begin{equation*}
  \Omega_\omega:=(-1,1)^2\cap
  \{x\in \R^2: (r(x),\varphi(x))\in(0,\sqrt{2}]\times[0,\omega]\},
\end{equation*}
where $r$ and $\varphi$ stand for the polar coordinates located at the
origin. The boundary of $\Omega_\omega$ is denoted by $\Gamma_\omega$
which is decomposed into straight line segments $\Gamma_j$,
$j=1,\ldots,m(\omega)$, counting counterclockwise beginning at the
origin.
As numerical example we consider the problem
\begin{equation*}
  \begin{aligned}
    -\Delta y   &= 0 && \text{in }\Omega_\omega,\\
    y &= u && \text{on }\Gamma_j, \quad j=1,\dots,m(\omega).
  \end{aligned}
\end{equation*}
The boundary datum $u$ is chosen as
\[
  u:=r^{-0.4999}\sin(-0.4999\varphi)\quad \text{on } \Gamma_\omega.
\]
This function belongs to $L^p(\Gamma)$ for every $p<2.0004$. The exact
solution of our problem is simply
\[
  y=r^{-0.4999}\sin(-0.4999\varphi),
\]
since $y$ is harmonic. A plot of it can be seen in Figure \ref{fig:1}.
\begin{figure}\centering
  \includegraphics[width=0.5\textwidth]{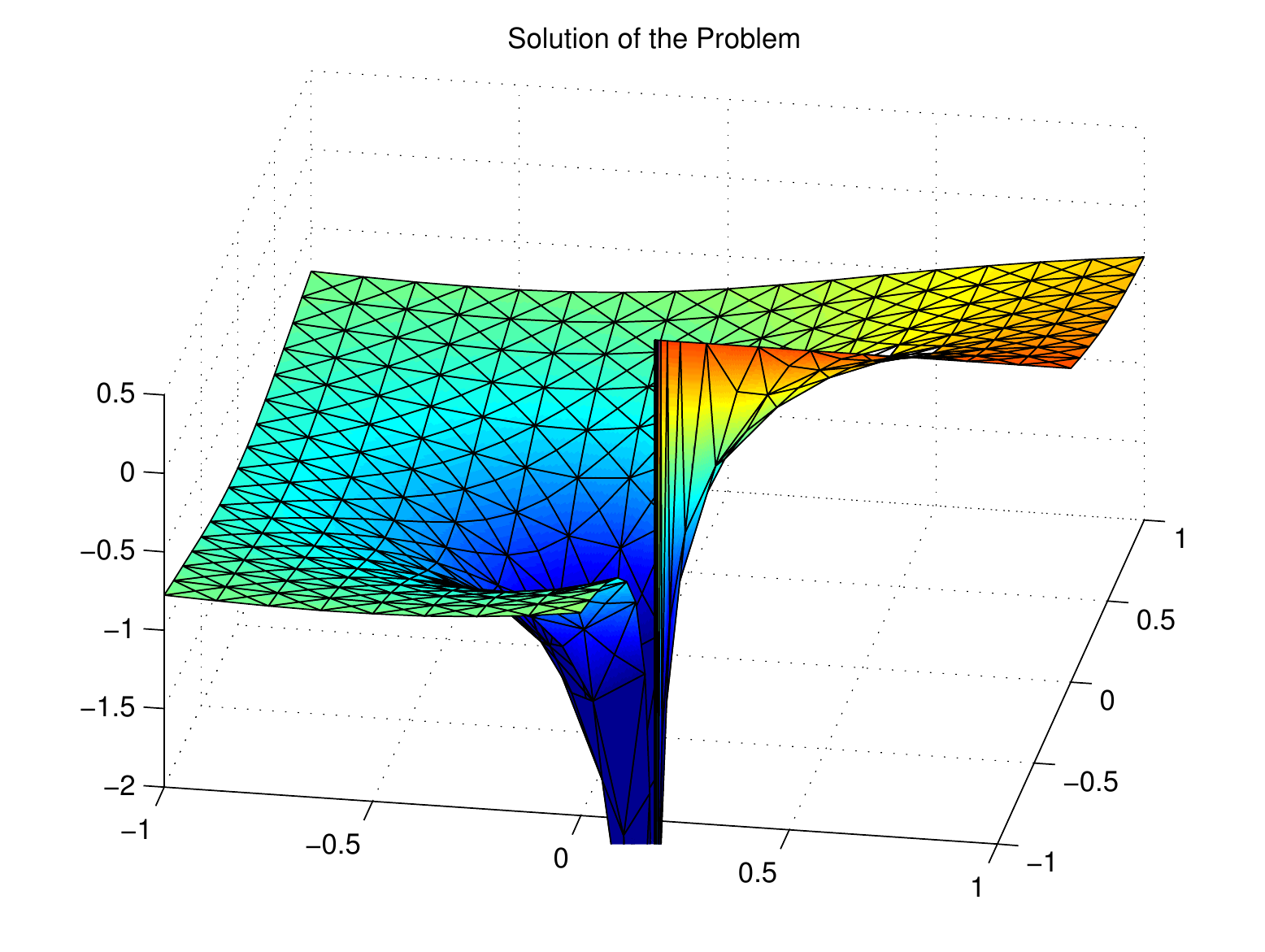}
  \caption{\label{fig:1}Visualization of the exact solution}
\end{figure}
We solve the problem numerically by using a finite element method with piecewise linear
finite elements combined with either the $L^2$-projection or the Carstensen interpolant
of the data on the boundary.
The finite element meshes for the calculations are generated by using
a newest vertex bisection algorithm as described in
\cite{pfefferer:14}. The discretization errors for different mesh
sizes and the experimental orders of convergence are given in Tables
\ref{tab:1}--\ref{tab:3} below for interior angles
$\omega\in\{3\pi/4,3\pi/2,355\pi/180\}$, where the discrete solutions based on the $L^2$-projection and the Carstensen
interpolant are denoted by $y_{h,2}$ or $y_{h,C}$, respectively. We note that the errors are calculated by an adaptive quadrature formula. Apparently, the results are very much in
congruence with the predicted orders.
\begin{table}[p]\centering
    \begin{tabular}{crcccc}
      \toprule
      mesh size $h$  & $\#$ unknowns & $\|y-y_{h,2}\|_{L^2(\Omega_\omega)}$ & eoc     & $\|y-y_{h,C}\|_{L^2(\Omega_\omega)}$ & eoc    \\
      \midrule
      0.50000        & 19            & 0.26142                          &         & 0.26794                          &        \\
      0.25000        & 61            & 0.18577                          & 0.49289 & 0.18973                          & 0.49794\\
      0.12500        & 217           & 0.13172                          & 0.49600 & 0.13426                          & 0.49899\\
      0.06250        & 817           & 0.09331                          & 0.49745 & 0.09497                          & 0.49940\\
      0.03125        & 3169          & 0.06605                          & 0.49838 & 0.06717                          & 0.49965\\
      0.01562        & 12481         & 0.04674                          & 0.49902 & 0.04750                          & 0.49982\\
      0.00781        & 49537         & 0.03306                          & 0.49942 & 0.03359                          & 0.49992\\
      0.00390        & 197377        & 0.02338                          & 0.49967 & 0.02375                          & 0.49998\\
      \midrule
      expected &&&0.5&&0.5 \\
      \bottomrule
    \end{tabular}
\caption{\label{tab:1}Discretization errors and experimental orders of convergence (eoc) for $\omega=3\pi/4$.
Expected convergence rate: $1/2$}
\end{table}%
\begin{table}[p]\centering
    \begin{tabular}{crcccc}
      \toprule
      mesh size $h$  & $\#$ unknowns & $\|y-y_{h,2}\|_{L^2(\Omega_\omega)}$ & eoc     & $\|y-y_{h,C}\|_{L^2(\Omega_\omega)}$ & eoc     \\
      \midrule                                                                    
      0.50000        & 33            & 0.73622                          &         & 0.77007                          &         \\
      0.25000        & 113           & 0.64484                          & 0.19118 & 0.67086                          & 0.19897 \\
      0.12500        & 417           & 0.56841                          & 0.18201 & 0.58915                          & 0.18737 \\
      0.06250        & 1601          & 0.50328                          & 0.17555 & 0.52022                          & 0.17950 \\
      0.03125        & 6273          & 0.44674                          & 0.17194 & 0.46091                          & 0.17464 \\
      0.01562        & 24833         & 0.39711                          & 0.16987 & 0.40920                          & 0.17166 \\
      0.00781        & 98817         & 0.35330                          & 0.16865 & 0.36376                          & 0.16982 \\
      0.00390        & 394241        & 0.31448                          & 0.16793 & 0.32362                          & 0.16868 \\
      \midrule
      expected &&&0.16667&&0.16667 \\
      \bottomrule
    \end{tabular}
\caption{\label{tab:2}Discretization errors and experimental orders of convergence (eoc) for $\omega=3\pi/2$.
Expected convergence rate: $1/6$}
\end{table}%
\begin{table}[p]\centering
    \begin{tabular}{crcccc}
      \toprule
      mesh size $h$  & $\#$ unknowns & $\|y-y_{h,2}\|_{L^2(\Omega_\omega)}$ & eoc     & $\|y-y_{h,C}\|_{L^2(\Omega_\omega)}$ & eoc    \\
      \midrule
      0.50000        & 46            & 1.1049                           &         & 1.1141                           &        \\
      0.25000        & 159           & 1.0693                           & 0.04721 & 1.0732                           & 0.05406\\
      0.12500        & 589           & 1.0491                           & 0.02749 & 1.0513                           & 0.02967\\
      0.06250        & 2265          & 1.0367                           & 0.01715 & 1.0384                           & 0.01782\\
      0.03125        & 8881          & 1.0281                           & 0.01207 & 1.0296                           & 0.01226\\
      0.01562        & 35169         & 1.0213                           & 0.00956 & 1.0228                           & 0.00962\\
      0.00781        & 139969        & 1.0154                           & 0.00832 & 1.0169                           & 0.00834\\
      0.00390        & 558465        & 1.0100                           & 0.00771 & 1.0114                           & 0.00772\\
      \midrule
      expected &&&0.00704&&0.00704 \\
      \bottomrule
    \end{tabular}
\caption{\label{tab:3}Discretization errors and experimental orders of convergence (eoc) for $\omega=355\pi/180$,
Expected convergence rate: $180/355-1/2=1/142\approx 0.007$}
\end{table}%
The different numerical solutions $y_{h,2}$ and $y_{h,C}$ for $\omega=3\pi/2$ and $h=1/8$ are displayed in
Figure \ref{fig:2a} and Figure \ref{fig:2b}, respectively.
\begin{figure}\centering
	\begin{subfigure}[c]{0.49\textwidth}
		\includegraphics[width=\textwidth]{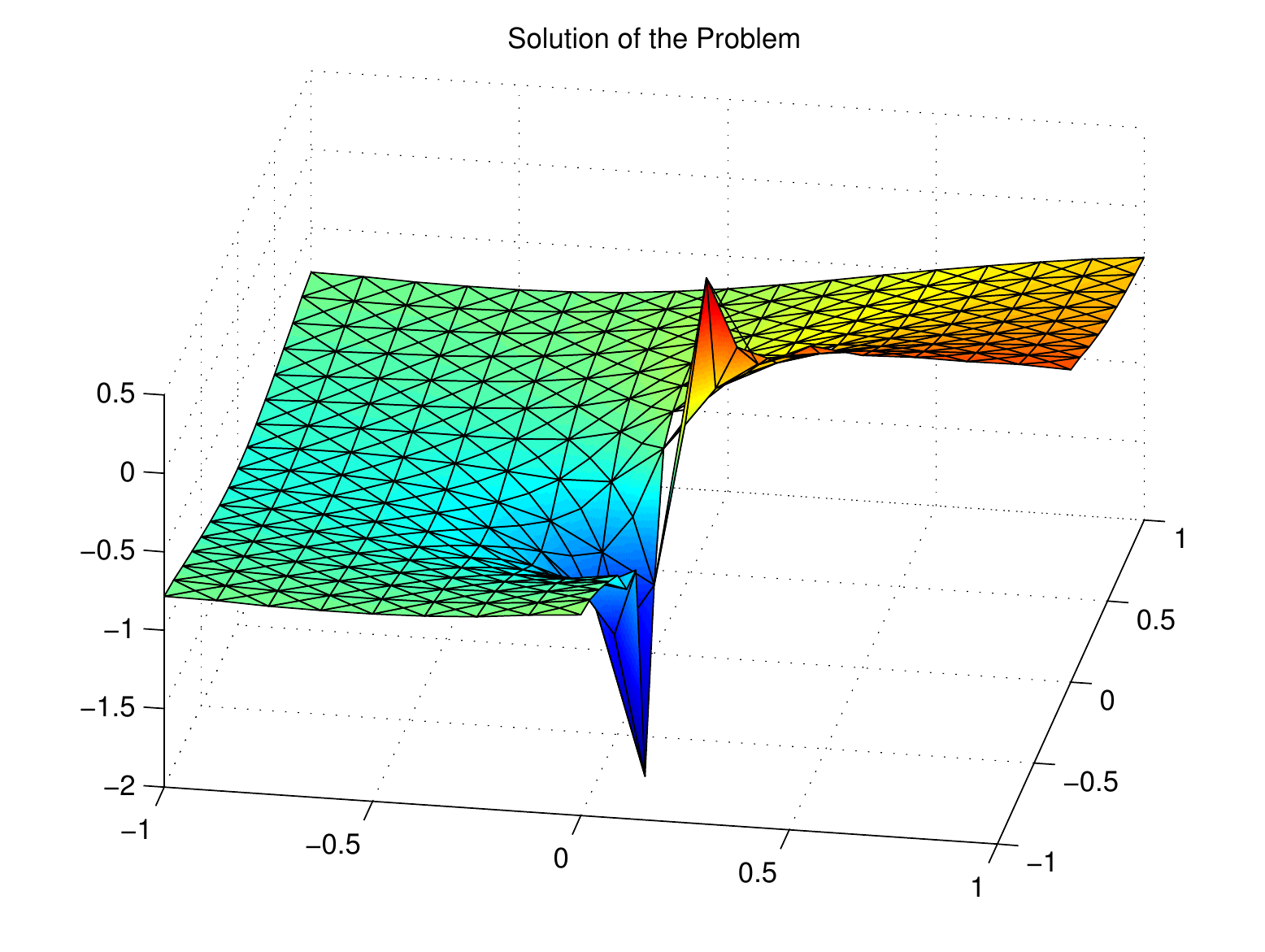}
		\caption{\label{fig:2a}Using $L^2$-projection: $u^h=\Pi_h u$}
  \end{subfigure}\hfill
  \begin{subfigure}[c]{0.49\textwidth}
		\includegraphics[width=\textwidth]{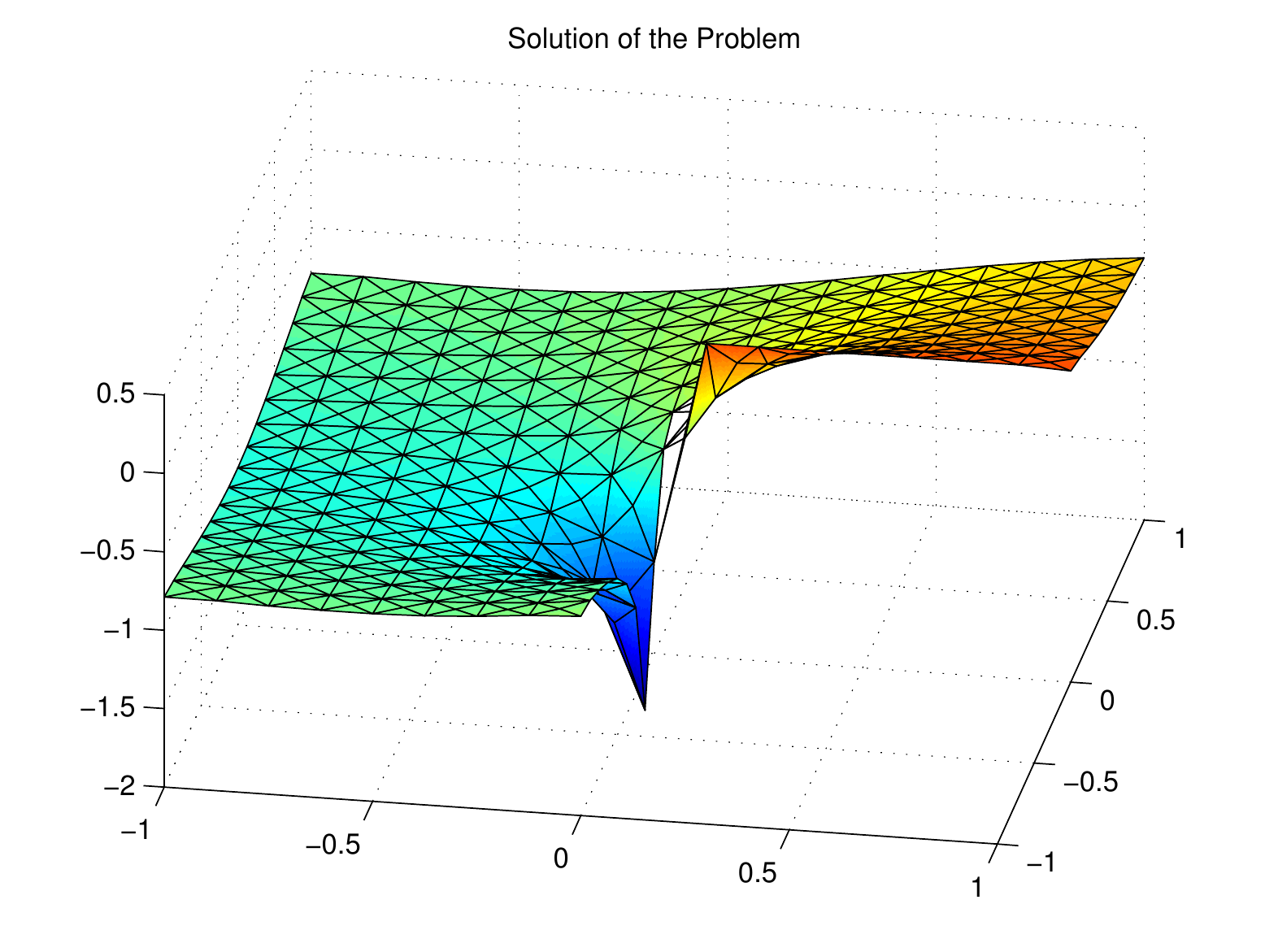}
		\caption{\label{fig:2b}Using Carstensen interpolant: $u^h=C_h u$}
  \end{subfigure}
  \caption{Visualization of the
    approximate solutions with $h=1/8$}
\end{figure}
We see that the infinite boundary
value in the origin is replaced by a finite one. If using the $L^2$-projection of the data, we see also that the
zero boundary values at the edge with $\theta=0$ are replaced by an
oscillating function which is typical for this kind of regularization.
By using the Carstensen interpolant as regularization of the data, this can be avoided according to the local definition of this interpolation operator. But we note
that these oscillations are a feature of the former regularization approach and do not disturb
the approximation order.

\section{Extensions}

\begin{remark}
  The paper is written for two-dimensional domains. However most
  results also hold in the three-dimensional case or can simply be extended to this one.
  The most crucial issue is the regularity which we have used for the corresponding adjoint problem since in three dimensional
  domains not only corner singularities but also edge singularities need to be taken into account.
\end{remark}

\begin{remark}\label{rem:5.2}
		We assume $f\in H^{-1}(\Omega)$ for the discretization error estimates of
		Corollary \ref{lem:globalerrorestconvex}. This is only for simplicity.
		For defining the very weak solution and the numerical solution we only need $f\in
    (H^1_0(\Omega)\cap H^1_\Delta(\Omega))'\cap Y_{0h}'$
    such that we could admit right hand sides from $L^1(\Omega)$ and Dirac
    measures as well. However, in this case the discretization error analysis
    demands an adapted proof, which exceeds the
    scope of this paper.
\end{remark}

\begin{remark} \label{rem:5.3} 
  We assume $u\in L^2(\Gamma)$ for simplicity. The case $u\in
  H^t(\Gamma)$, $t>0$, 
  is also of interest in the analysis of Dirichlet control problems,
  see \cite{ApelMateosPfeffererRoesch2013}.  The results can be
  improved in this more regular case: With $s$ from Corollary
  \ref{lem:globalerrorestconvex} we have
   \begin{align*}
     \|y-y^h\|_{L^2(\Omega)} \le 
     c\|u-u^h\|_{H^{-s}(\Gamma)} \le 
     ch^{s+t}\|u\|_{H^t(\Gamma)}, \quad t\in[0,\tfrac12].
   \end{align*}
   The first step is again the application of Lemma
   \ref{lem:regularizationerror}, while the second can be proved in
   analogy to Lemma \ref{lem:regularizationerrorforu}. The necessary
   prerequisites are already provided in Lemma \ref{lem:carstensen}
   and Remark \ref{rem:errestL2proj}. Furthermore, when we check the
   proof of Lemma \ref{lem:discrerrorestconvex} we find that we obtain
   \[
     \|y_0^h-y_{0h}\|_{L^2(\Omega)}\le ch^s 
     \left(h^{1/2}\|f\|_{H^{-1}(\Omega)}+
     h^t\| u^h\|_{H^t(\Gamma)}\right), \quad t\in[0,\tfrac12].
   \]
   Hence it remains to prove 
   \begin{align}\label{Htstability}
     \| u^h\|_{H^t(\Gamma)}\le c\| u\|_{H^t(\Gamma)}, 
   \end{align}
   in order to conclude 
   \[
     \|y-y_h\|_{L^2(\Omega)}\le ch^{s+t}
     \left(h^{1/2-t}\|f\|_{H^{-1}(\Omega)}+\|u\|_{H^t(\Gamma)}\right), 
     \quad t\in[0,\tfrac12].
   \]

   The estimate \eqref{Htstability} is known for domains and can be
   proved for $t\in [0,1]$ also for the boundary $\Gamma$ by using the
   inverse inequality and the approximation properties of $u^h$ and
   the Ritz projection. To this end recall that $\Gamma_j$,
   $j=1,\ldots,N$, are the boundary segments of $\Gamma$, and let
   $P_h^j$ be the Ritz projection on $Y_h^\partial|_{\Gamma_j}$. Then
   we have
   \begin{align*}
     \|u^h\|_{H^1(\Gamma)} &\le \sum_{j=1}^N\|u^h-P_h^ju\|_{H^1(\Gamma_j)} +
     \sum_{j=1}^N\|P_h^ju\|_{H^1(\Gamma_j)} \\ &\le
     ch^{-1} \sum_{j=1}^N\|u^h-P_h^ju\|_{L^2(\Gamma_j)} +
     \sum_{j=1}^N\|u\|_{H^1(\Gamma_j)} \\ &\le
     ch^{-1} \|u-u^h\|_{L^2(\Gamma)} + 
     ch^{-1} \sum_{j=1}^N \|u-P_h^ju\|_{L^2(\Gamma_j)} +
     \|u\|_{H^1(\Gamma)} \\ &\le c\|u\|_{H^1(\Gamma)},
   \end{align*}
   i.~e., we have \eqref{Htstability} for $t=1$.  Since we proved
   \eqref{Htstability} for $t=0$ in Lemma
   \ref{lem:regularizationerrorforu} we get the desired result by
   interpolation in Sobolev spaces. 
\end{remark}

\addtocounter{section}{1}
\renewcommand{\thesection}{A}
\setcounter{theorem}{0}
\section*{Appendix: Error estimates for the Carstensen interpolant}

Recall from Subsection \ref{sec:regularizationstrategy} that the
piecewise linear Carstensen interpolant is defined via
\[
  C_hu=\sum_{x\in {\mathcal N}_{\Gamma}} \pi_x(u)\lambda_x,\quad
  \pi_x(u)=\frac{\int_{\Gamma} u\lambda_x} {\int_{\Gamma} \lambda_x}
  =\frac{(u,\lambda_x)_\Gamma}{(1,\lambda_x)_\Gamma}.
\]
where $\lambda_x$ is the standard hat function related to $x$.

\begin{lemma}\label{lem:ChIntErrEst}
  The piecewise linear Carstensen interpolant satisfies the error estimate
  \begin{align}\label{ChIntErrEst}
    \|\varphi-C_h\varphi\|_{L^2(\Gamma)}\leq ch^s \|\varphi\|_{H^s(\Gamma)},
  \end{align}
  for $\varphi\in H^s(\Gamma)$, $s\in[0,1]$.
\end{lemma}
\begin{proof}
  The interpolation error estimate
  \eqref{ChIntErrEst} is in principle contained in
  \cite{Carstensen1999}, however, it is there an estimate on a domain
  such that we sketch the proof here. 

  For $s=0$ this estimate follows from the stability property
  \begin{align}\label{stabCarst}
    \|\pi_x(u)\|_{L^2(\omega_x)}\le c \|u\|_{L^2(\omega_x)},
  \end{align}
  where $\omega_x$ is the support of $\lambda_x$ on $\Gamma$.

  For $s=1$ we use that $\pi_x(w)=w$ for all constants $w$ such that
  \[ 
    \|u-\pi_x(u)\|_{L^2(\omega_x)} = \|(u-w)-\pi_x(u-w)\|_{L^2(\omega_x)} \le
    c\|u-w\|_{L^2(\omega_x)} \le ch \|u\|_{H^1(\omega_x)}
  \]
  via the Deny--Lions lemma, see also \cite[Lemma
  4.3]{ReyesMeyerVexler2008} where the piecewise affine and Lipschitz
  continuous transformation of $\omega_x$ to some reference domain is
  discussed in detail. For a boundary edge with end points $x_1$ and
  $x_2$ we have
  \[
    \|u-C_hu\|_{L^2(e)} = 
    \|(u-\pi_{x_1}(u))\lambda_{x_1} + (u-\pi_{x_2}(u))\lambda_{x_2}\|_{L^2(e)} \le c
    \sum_{i=1}^2 \|u-\pi_{x_i}(u)\|_{L^2(e)}.
  \]
  From these two estimates we obtain \eqref{ChIntErrEst} in the case
  $s=1$. In the remaining case $s\in(0,1)$ the error
  estimate \eqref{ChIntErrEst} follows by interpolation of Sobolev
  spaces.
\end{proof}

\begin{lemma}\label{lem:carstensen}
  The piecewise linear Carstensen interpolant satisfies the error estimate 
  \[
    \|u-C_hu\|_{H^{-s}(\Gamma)}\le ch^{s+t}\|u\|_{H^t(\Gamma)},
    \quad s\in[0,1],\quad t\in[0,1],
  \]
  as well as the stability estimate
  \[
    \|C_hu\|_{L^2(\Gamma)}\le c\|u\|_{L^2(\Gamma)}.
  \]
\end{lemma}

\begin{proof}
  The second estimate follows directly from the fact that $0\leq
  \lambda_x\leq 1$ and the stability property \eqref{stabCarst}.

  The first estimate is in principle contained in
  \cite{ReyesMeyerVexler2008}, however, it is there an estimate on a
  domain such that we sketch the proof here.  First we notice that the
  definition of $\pi_x(u)$ is equivalent to
  $(u-\pi_x(u),\lambda_x)_\Gamma=0$ and hence we have
  \[
    (u-\pi_x(u),\pi_x(\varphi)\lambda_x)_\Gamma=0 \quad\forall\varphi\in H^s(\Gamma).
  \]
  With this identity and with $\sum_x\lambda_x=1$ we get
  \begin{align*}
    (u-C_hu, \varphi)_\Gamma &=
    (u\sum_x\lambda_x-\sum_x \pi_x(u)\lambda_x, \varphi)_\Gamma =
    \sum_x (u-\pi_x(u), \varphi\lambda_x)_\Gamma \\ &=
    \sum_x (u-\pi_x(u), (\varphi-\pi_x(\varphi))\lambda_x)_\Gamma \\ &\le
    \sum_x \|u-\pi_x(u)\|_{L^2(\omega_x)} \|\varphi-\pi_x(\varphi)\|_{L^2(\omega_x)}
  \end{align*}
  where $\omega_x$ is again the support of $\lambda_x$ on $\Gamma$.
  With similar arguments as in the proof of Lemma \ref{lem:ChIntErrEst} we conclude
  \[
    (u-C_hu, \varphi)_\Gamma\le ch^{s+t}\|u\|_{H^t(\Gamma)}\|\varphi\|_{H^s(\Gamma)}.
  \]
  By the definition of the negative norm,
  \[
    \|u-C_hu\|_{H^{-s}(\Gamma)}=
    \sup_{\varphi\in H^{s}(\Gamma), \varphi\ne 0}
    \frac{(u-C_hu, \varphi)_\Gamma}{\|\varphi\|_{H^{s}(\Gamma)}},
  \]
  we obtain the assertion of the lemma.  
\end{proof}

\begin{remark}\label{rem:errestL2proj}
  Note that this error estimate holds also for the $L^2$-projection,
  \[
    \|u-\Pi_hu\|_{H^{-s}(\Gamma)}\le ch^{s+t}\|u\|_{H^t(\Gamma)},
    \quad s\in[0,1],\quad t\in[0,1].
  \]
  It can be proved  similarly by using
  $(u-\Pi_hu, \varphi)_\Gamma=(u-\Pi_hu, \varphi-\Pi_h\varphi)_\Gamma$.
\end{remark}

\bigskip
\begin{acknowledgement}
  The authors thank Markus Melenk and Christian Simader for helpful
  discussions.  
\end{acknowledgement}

\bibliographystyle{abbrv}\bibliography{ApPf}
\end{document}